\documentclass[11pt]{article}
\usepackage{epsf}

\usepackage{amsthm}
\usepackage{amssymb}
\usepackage{amsmath}

\usepackage[top=1.5cm,bottom=1.5cm,left=2.5cm,right=2.5cm,includehead,includefoot
           ]{geometry}

\makeatletter

\numberwithin{equation}{section}
 \newtheorem{theorem}{Theorem}[section]
\newtheorem{lemma}[theorem]{Lemma}
\newtheorem{corollary}[theorem]{Corollary}
\newtheorem{remark}[theorem]{Remark}

\title{On (\textit{m,n,l})-Jordan Centralizers of Some Algebras}
\author{\begin{tabular}{c}Jianbin Guo, Jiankui Li\footnote{Corresponding author.
E-mail address: jiankuili@yahoo.com} ~and Qihua Shen\\
{\small\it Department of Mathematics, East China University of
Science and Technology}\\
{\small\it Shanghai 200237, P. R. China}
\end{tabular}}
\date{}
\begin{document}
\maketitle \abstract
Let $\mathcal{A}$ be a unital algebra over a number field $\mathbb{K}$. A linear mapping $\delta$ from $\mathcal{A}$ into itself is called a weak (\textit{m,n,l})-Jordan centralizer
if $(m+n+l)\delta(A^2)-m\delta(A)A-nA\delta(A)-lA\delta(I)A\in \mathbb{K}I$ for every $A\in \mathcal{A}$,
where $m\geq0, n\geq0, l\geq0$ are fixed integers with $m+n+l\neq 0$.
In this paper, we study weak (\textit{m,n,l})-Jordan centralizer on generalized matrix algebras and some reflexive algebras alg$\mathcal{L}$, where $\mathcal{L}$ is a CSL or satisfies $\vee\{L: L\in \mathcal{J}(\mathcal{L})\}=X$ or $\wedge\{L_-: L\in \mathcal{J}(\mathcal{L})\}=(0)$, and prove that each weak (\textit{m,n,l})-Jordan centralizer of these algebras is a centralizer when $m+l\geq1$ and $n+l\geq1$.

\


\

{\sl 2000 AMS classification} : Primary 47L35; Secondly 17B40, 17B60
\

\section{Introduction}\

Let $\mathcal{A}$ be an algebra over a number field $\mathbb{K}$ and
$\mathcal{M}$ be an $\mathcal{A}$-bimodule. An additive (linear)
mapping $\delta: \mathcal{A}\rightarrow \mathcal{M}$ is called a
\textit{left (right) centralizer} if $\delta(AB)=\delta(A)B$
($\delta(AB)=A\delta(B)$) for all $A, B\in \mathcal{A}$. An additive
(linear) mapping $\delta: \mathcal{A}\rightarrow \mathcal{M}$ is
called a \textit{left (right) Jordan centralizer} if
$\delta(A^2)=\delta(A)A$ ($\delta(A^2)=A\delta(A)$) for every $A\in
\mathcal{A}$. We call $\delta$ a \textit{centralizer} if $\delta$ is
both a left centralizer and a right centralizer. Similarly, we can
define a \textit{Jordan centralizer}. It is clear that every
centralizer is a Jordan centralizer, but the converse is not true in
general. In \cite{Zalar}, Zalar proved that each left Jordan
centralizer of a semiprime ring is a left centralizer and each
Jordan centralizer of a semiprime ring is a centralizer. For some
other results, see \cite{CSR,CORA,OCS} and references therein.\

In \cite{mn}, Vukman defined a new type of Jordan centralizers, named (\textit{m,n})-\textit{Jordan centralizer}, that is
an additive mapping $\delta$ from a ring $\mathcal{R}$ into itself such that $$(m+n)\delta(x^2)=m\delta(x)x+nx\delta(x)$$ for every $x\in \mathcal{R}$,
where $m\geq0$, $n\geq0$ are fixed integers with $m+n\neq0$.
Obviously, (1,0)-Jordan centralizer is a left Jordan centralizer and (0,1)-Jordan centralizer is a right Jordan centralizer.
Moreover, each Jordan centralizer is an ($m, n$)-Jordan centralizer and
$(1, 1)$-Jordan centralizer satisfies the relation $2\delta(x^2)=\delta(x)x+x\delta(x)$ for every $x\in \mathcal{R}$.
In \cite{CSR}, Vukman showed that $(1, 1)$-Jordan centralizer of a 2-torsion free semiprime ring $\mathcal{R}$ is a centralizer.
In \cite{submit}, Guo and Li studied $(1, 1)$-Jordan centralizer of some reflexive algebras.
In \cite{mn}, Vukman investigated ($m, n$)-Jordan centralizer and proved that for $m\geq1$ and $n\geq1$, every $(m, n)$-Jordan centralizer of a prime ring $\mathcal{R}$ with $char(\mathcal{R})\neq6mn(m+n)$ is a centralizer.
Motivated by this, we define a new type of Jordan centralizer, named \textit{weak }(\textit{m,n,l})-\textit{Jordan centralizer},
that is a linear mapping $\delta$ from a unital algebra $\mathcal{A}$ into itself satisfying
$$(m+n+l)\delta(A^2)-m\delta(A)A-nA\delta(A)-lA\delta(I)A\in\mathbb{K}I$$
for every $A\in \mathcal{A}$, where $m\geq0, n\geq0, l\geq0$ are fixed integers with $m+n+l\neq 0$.
This is equivalent to say that for every $A\in \mathcal{A}$, there exists a $\lambda_A\in \mathbb{K}$ such that
\begin{eqnarray}
(m+n+l)\delta(A^2)=m\delta(A)A+nA\delta(A)+lA\delta(I)A+\lambda_AI.\label{101}
\end{eqnarray}
When $\lambda_A=0$ for every $A\in \mathcal{A}$, we call such a $\delta$ an (\textit{m,n,l})-\textit{Jordan centralizer}.
It is clear that each (\textit{m,n,l})-Jordan centralizer is a weak (\textit{m,n,l})-Jordan centralizer, each (\textit{m,n,\emph{0}})-Jordan centralizer is an (\textit{m,n})-Jordan centralizer and (0,0,1)-Jordan centralizer has
the relation $\delta(A^2)=A\delta(I)A$ for every $A\in \mathcal{A}$. In this paper, we study (weak) (\textit{m,n,l})-Jordan centralizer on some reflexive algebras and generalized matrix algebras.

Let $X$ be a Banach space over $\mathbb{K}$ and $B(X)$ be the set of all bounded operators on $X$, where $\mathbb{K}$ is the real field $\mathbb{R}$ or the complex field $\mathbb{C}$.
We use $X^*$ to denote the set of all bounded linear functionals on $X$.
For $A\in B(X)$, denote by $A^*$ the adjoint of $A$.
For any non-empty subset $L\subseteq X$, $L^\perp$ denotes its annihilator, that is,
$L^\perp=\{f\in X^*: f(x)=0~\mathrm{for}~\mathrm{all}~x\in L\}$.
By a subspace lattice on $X$, we mean a collection $\mathcal{L}$ of closed subspaces of $X$
with (0) and $X$ in $\mathcal{L}$ such that for every family $\{M_r\}$ of elements of $\mathcal{L}$,
both $\cap M_r$ and $\vee M_r$ belong to $\mathcal{L}$. For a subspace lattice $\mathcal{L}$ of $X$,
let alg$\mathcal{L}$ denote the algebra of all operators in $B(X)$ that leave members of $\mathcal{L}$ invariant;
and for a subalgebra $\mathcal{A}$ of $B(X)$, let lat$\mathcal{A}$ denote the lattice of all closed subspaces of $X$
that are invariant under all operators in $\mathcal{A}$.
An algebra $\mathcal{A}$ is called \textit{reflexive} if alglat$\mathcal{A}=\mathcal{A}$;
and dually, a subspace lattice is called \textsl{reflexive} if latalg$\mathcal{L}=\mathcal{L}$.
Every reflexive algebra is of the form alg$\mathcal{L}$ for some subspace lattice $\mathcal{L}$ and vice versa.

For a subspace lattice $\mathcal{L}$ and for $E\in \mathcal{L}$, define
$$E_-=\vee\{F\in \mathcal{L}: F\nsupseteq E\}~\mathrm{and}~E_+=\wedge\{F\in \mathcal{L}: F\nleq E\}.$$
Put
$$\mathcal{J}(\mathcal{L})=\{K\in \mathcal{L}: K\neq(0)~\mathrm{and}~K_-\neq X\}.$$
For any non-zero vectors $x\in X$ and $f\in X^*$, the rank one operator $x\otimes f$ is defined by $x\otimes f(y)=f(y)x$ for $y\in X$.
Several authors have studied the properties of the set of rank one operators in reflexive algebras (for example, see \cite{lambrou, ORO}).
It is well known (see \cite{ORO}) that $x\otimes f\in \textrm{alg}\mathcal{L}$ if and only if there exists some $K\in \mathcal{J}(\mathcal{L})$
such that $x\in K$ and $f\in K_-^\perp$.
When $X$ is a separable Hilbert space over the complex field $\mathbb{C}$, we change it to $H$. In a Hilbert space,
we disregard the distinction between a closed subspace and the orthogonal projection onto it.
A subspace lattice $\mathcal L$ on a Hilbert space $H$ is called a \textit{commutative~subspace~lattice}
(\textit{CSL}), if all projections in $\mathcal L$ commute pairwise. If $\mathcal{L}$ is a CSL, then the
corresponding algebra alg$\mathcal{L}$ is called a \textit{CSL algebra}. By \cite{CSL}, we know that if $\mathcal{L}$ is a CSL, then $\mathcal{L}$
is reflexive.
Let $\mathcal{L}$ be a subspace lattice on a Banach space $X$ satisfying
$\vee\{L: L\in \mathcal{J}(\mathcal{L})\}=X$ or $\wedge\{L_-: L\in \mathcal{J}(\mathcal{L})\}=(0)$.
In \cite{JDRA}, Lu considered this kind of reflexive algebras which have rich rank one operators.
In Section 2, we prove that if $\delta$ is a weak (\textit{m,n,l})-Jordan centralizer from alg$\mathcal{L}$ into itself, where $\mathcal{L}$ is a CSL or satisfies $\vee\{L: L\in \mathcal{J}(\mathcal{L})\}=X$ or $\wedge\{L_-: L\in \mathcal{J}(\mathcal{L})\}=(0)$,
then $\delta$ is a centralizer.

A \textit{Morita context} is a set $(\mathcal{A}, \mathcal{B}, \mathcal{M}, \mathcal{N})$
and two mappings $\phi$ and $\varphi$, where $\mathcal{A}$ and $\mathcal{B}$ are two algebras over a number field $\mathbb{K}$, $\mathcal{M}$ is an
$(\mathcal{A}, \mathcal{B})$-bimodule and $\mathcal{N}$ is a $(\mathcal{B}, \mathcal{A})$-bimodule.
The mappings $\phi: M\otimes_\mathcal{B} N\rightarrow A$ and $\varphi: N\otimes_\mathcal{A} M\rightarrow B$ are two bimodule homomorphisms satisfying
$\phi(M\otimes N)M'=M\varphi(N\otimes M')$ and $\varphi(N\otimes M)N'=N\phi(M\otimes N')$ for any $M, M'\in \mathcal{M}$ and
$N, N'\in \mathcal{N}$. These conditions insure that the set
$$\left[
  \begin{array}{cc}
    \mathcal{A} & \mathcal{M} \\
    \mathcal{N} & \mathcal{B} \\
  \end{array}
\right]=\left\{\left[
            \begin{array}{cc}
              A& M \\
              N & B \\
            \end{array}
          \right]\mid ~A\in \mathcal{A}, M\in \mathcal{M}, N\in \mathcal{N}, B\in \mathcal{B}\right\}$$
form an algebra over $\mathbb{K}$ under usual matrix operations.
We call such an algebra a \textit{generalized matrix algebra} and denoted by
$\mathcal{U}=\left[
  \begin{array}{cc}
    \mathcal{A} & \mathcal{M} \\
    \mathcal{N} & \mathcal{B} \\
  \end{array}
\right]$, where $\mathcal{A}$ and $\mathcal{B}$ are two unital algebras and at least one of the two bimodules $\mathcal{M}$ and $\mathcal{N}$
is distinct from zero.  This kind of algebra was first introduced
by Sands in \cite{sands}.
Obviously, when $\mathcal{M}=0$ or $\mathcal{N}=0$, $\mathcal{U}$ degenerates to the triangular algebra.
In Section 3, we show that if $\delta$ is a weak (\textit{m,n,l})-Jordan centralizer
from $\mathcal{U}$ into itself, then $\delta$ is a centralizer. We also study $(m,n,l)$-Jordan centralizer on AF algebras.
Throughout the paper, we assume $m,n,l\in \mathbb{N}$ are such that $m+l\geq1$, $n+l\geq1$.

\section{Centralizers of certain reflexive algebras}
\

In order to prove our main results, we need the following  several lemmas.

\begin{lemma}\label{201}
Let $\mathcal{A}$ be a unital algebra with identity $I$ and $\delta$ be a weak \emph{(}\textit{m,n,l}\emph{)}-Jordan centralizer
from $\mathcal{A}$ into itself. Then for any $A,B\in\mathcal{A}$,
\begin{eqnarray}
&&(m+n+l)\delta(AB+BA)=m\delta(A)B+m\delta(B)A+nA\delta(B)+nB\delta(A)\nonumber\\
&&~~~~~~~~~~~~~~~~~~~~~~~~~~~~~~~~~+lA\delta(I)B+lB\delta(I)A+(\lambda_{A+B}-\lambda_A-\lambda_B)I.\label{2001}
\end{eqnarray}
In particular, for any $A\in\mathcal{A}$,
\begin{eqnarray}
\delta(A)=\frac{m+l}{m+n+2l}\delta(I)A+\frac{n+l}{m+n+2l}A\delta(I)+\lambda(A),\label{2002}
\end{eqnarray}
where we set $\lambda(A)=\frac{1}{m+n+2l}(\lambda_{A+I}-\lambda_A)I$ for every $A \in \mathcal{A}$.
\end{lemma}

\begin{proof}
Since $\delta$ is a weak (\textit{m,n,l})-Jordan centralizer, we have
\begin{eqnarray*}
(m+n+l)\delta(A^2)=m\delta(A)A+nA\delta(A)+lA\delta(I)A+\lambda_AI
\end{eqnarray*}
for every $A\in \mathcal{A}.$
Replacing $A$ by $A+B$ in above equation, (\ref{2001}) holds. Letting $B=I$ in (\ref{2001}) gives (\ref{2002}), since $\lambda_I=0$.
\end{proof}

\begin{remark}\label{211}
\emph{For an (\textit{m,n,l})-Jordan centralizer, we could actually define it from a unital algebra $\mathcal{A}$ to an $\mathcal{A}$-bimodule. Hence when lemmas in this section are applied to an (\textit{m,n,l})-Jordan centralizer $\delta$, we will take it for granted that $\delta$ is from a unital algebra $\mathcal{A}$ to its bimodule, since all the proofs remain true if we set $\lambda_A=0$ for all $A\in\mathcal{A}$.}
\end{remark}
\begin{remark} \label{212}
\emph{Obviously, each (1,0,0)-Jordan centralizer is a left Jordan centralizer and each (0,1,0)-Jordan centralizer is a right Jordan centralizer.
So by Lemma \ref{201}, it follows that every left Jordan centralizer of unital algebras is a left centralizer
and every right Jordan centralizer of unital algebras is a right centralizer.
Therefore every Jordan centralizer of unital algebras is a centralizer.}
\end{remark}
Let $f$ be a linear mapping from an algebra $\mathcal{A}$ to its bimodule $\mathcal{M}$. Recall that $f$ is a \emph{derivation} if $f(ab)=f(a)b+af(b)$ for all $a,b\in\mathcal{A}$; it is a \emph{Jordan derivation} if $f(ab)=f(a)a+af(a)$ for every $a\in \mathcal{A}$; it is a \emph{generalized derivation} if $f(ab)=f(a)b+ad(b)$ for all $a,b\in\mathcal{A}$, where $d$ is a derivation from $\mathcal{A}$ to $\mathcal{M}$; and it is a \emph{generalized Jordan derivation} if $f(ab)=f(a)a+ad(a)$ for every $a\in\mathcal{A}$, where $d$ is a Jordan derivation from $\mathcal{A}$ to $\mathcal{M}$.
From Remarks \ref{211} and \ref{212}, we have the following corollary.
\begin{corollary}\label{202}
Let $\mathcal{L}$ be a subspace lattice on a Banach space $X$ satisfying $\vee\{F: F\in \mathcal{J}(\mathcal{L})\}=X$ or
$\wedge\{L_-: L\in \mathcal{J}(\mathcal{L})\}=(0)$.
If $f$ is a generalized Jordan derivation from alg$\mathcal{L}$ to $B(X)$, then $f$ is a generalized derivation.
\end{corollary}
\begin{proof}
Since $f$ is a generalized Jordan derivation, we have the relation
\begin{eqnarray*}
f(A^2)=f(A)A+Ad(A)
\end{eqnarray*}
for every $A\in \textrm{alg}\mathcal{L}$, where $d$ is a Jordan derivation of alg$\mathcal{L}$.
By \cite[Theorem 2.1]{JDRA}, one can conclude that $d$ is a derivation. Let $\delta=f-d$. Then we have
$$\delta(A^2)=f(A^2)-d(A^2)=f(A)A+Ad(A)-Ad(A)-d(A)A=f(A)A-d(A)A=\delta(A)A$$
for every $A\in \textrm{alg}\mathcal{L}$. This means that $\delta$ is a left Jordan centralizer. By
Remark \ref{212}, $\delta$ is a left centralizer. Hence
$$f(AB)=d(AB)+\delta(AB)=d(A)B+Ad(B)+\delta(A)B=f(A)B+Ad(B)$$
for all $A, B\in \textrm{alg}\mathcal{L}$. In other words, $f$ is a generalized derivation.
\end{proof}

Since every Jordan derivation of CSL algebras is a derivation \cite{JDCSL}, we also have the following corollary.

\begin{corollary}\label{203}
Let $\mathcal{L}$ be a \textit{CSL} on a Hilbert space $H$.
If $f$ is a generalized Jordan derivation from alg$\mathcal{L}$ into itself, then $f$ is a generalized derivation.
\end{corollary}

\begin{lemma}\label{204}
Let $\mathcal{A}$ be a unital algebra
and $\delta$ be a weak \emph{(}\textit{m,n,l}\emph{)}-Jordan centralizer from $\mathcal{A}$ into itself.
Then for every idempotent  $P \in \mathcal{A}$ and every $A\in\mathcal{A}$, \\
\emph{(i)}~~$\delta(P)=P\delta(I)=\delta(I)P$;\\
\emph{(ii)}~~$\delta(AP)=\delta(A)P+\lambda(AP)-\lambda(A)P$;\\
\emph{(iii)}~~$\delta(PA)=P\delta(A)+\lambda(PA)-\lambda(A)P$.
\end{lemma}

\begin{proof}
(i) Suppose $P$ is an idempotent in $\mathcal{A}$. It follows from Lemma \ref{201} that
\begin{eqnarray}
(m+n+2l)\delta(P)=(m+l)\delta(I)P+(n+l)P\delta(I)+(\lambda_{P+I}-\lambda_P)I.\label{2003}
\end{eqnarray}
Right and left multiplication of (\ref{2003}) by $P$ gives $$P\delta(P)P=P\delta(I)P+\frac{1}{m+n+2l}(\lambda_{P+I}-\lambda_P)P.$$
Since $(m+n+l)\delta(P)=m\delta(P)P+nP\delta(P)+lP\delta(I)P+\lambda_PI$,
multiplying $P$ from the right leads to
\begin{eqnarray*}
(n+l)\delta(P)P&=&n(P\delta(I)P+\frac{1}{m+n+2l}(\lambda_{P+I}-\lambda_P)P)+lP\delta(I)P+\lambda_PP\\
&=&(n+l)P\delta(I)P+(\frac{n}{m+n+2l}(\lambda_{P+I}-\lambda_P)+\lambda_P)P,
\end{eqnarray*}
whence
\begin{eqnarray}
\delta(P)P=P\delta(I)P+\varepsilon_PP\label{2007}
 \end{eqnarray}
for some $\varepsilon_P\in \mathbb{C}$.
Similarly, $P\delta(P)=P\delta(I)P+\varepsilon_P'P$ for some $\varepsilon_P'\in \mathbb{C}$. \\
Hence $\delta(P)P-\varepsilon_PP=P\delta(P)-\varepsilon_P'P$. Right and left multiplication of $P$ gives $\varepsilon_P=\varepsilon_P'$, which implies \begin{eqnarray}
\delta(P)P=P\delta(P).\label{2012}
 \end{eqnarray}

Replacing $P$ by $I-P$ in the above equation gives $\delta(I)P=P\delta(I)$.

Now, we have from (\ref{2003})
\begin{eqnarray}
\delta(P)=\delta(I)P+\frac{1}{m+n+2l}(\lambda_{P+I}-\lambda_P)I.\label{2008}
\end{eqnarray}
On the other hand, (\ref{2007}) and (\ref{2012}) yields
\begin{eqnarray*}
(m+n+l)\delta(P)&=&m\delta(P)P+nP\delta(P)+lP\delta(I)P+\lambda_PI\\
&=&(m+n+l)\delta(P)P+\lambda_PI-l\varepsilon_PP,
\end{eqnarray*}
right multiplication of which by $P$ gives $\lambda_P=l\varepsilon_P$. Hence
\begin{eqnarray}
\delta(P)=\delta(P)P+\frac{1}{m+n+l}\lambda_P(I-P)\label{2006}.
\end{eqnarray}
We then have from (\ref{2008}) that
\begin{eqnarray}
\delta(P)P=\delta(I)P+\frac{1}{m+n+2l}(\lambda_{P+I}-\lambda_{P})P \label{2014}
\end{eqnarray}
 and hence
\begin{eqnarray*}
\delta(P)=\delta(I)P+\frac{1}{m+n+2l}(\lambda_{P+I}-\lambda_{P})P+\frac{1}{m+n+l}\lambda_P(I-P),
\end{eqnarray*}
which together with (\ref{2008}) implies
$$\frac{1}{m+n+2l}(\lambda_{P+I}-\lambda_{P})=\frac{1}{m+n+l}\lambda_{P}.$$
Thus we have
\begin{eqnarray}
\delta(P)=\delta(I)P+\frac{1}{m+n+l}\lambda_{P}I,\label{2009}
\end{eqnarray}
while
\begin{eqnarray}
\delta(P)=\frac{m+n}{m+n+l}\delta(P)P+\frac{l}{m+n+l}\delta(I)P+\frac{1}{m+n+l}\lambda_{P}I. \label{2010}
\end{eqnarray}
Comparing (\ref{2009}) and (\ref{2010}) gives $$\delta(I)P=\delta(P)P.$$
This together with (\ref{2014}) gives
$$\lambda(P)=\frac{1}{m+n+2l}(\lambda_{P+I}-\lambda_{P})I=\frac{1}{m+n+l}\lambda_{P}I=0,$$ whence
$$\delta(P)=\delta(I)P=P\delta(I).$$
(ii) By Lemma \ref{201} and (i), we have
\begin{eqnarray*}
\delta(AP)&=&\frac{m+l}{m+n+2l}\delta(I)AP+\frac{n+l}{m+n+2l}AP\delta(I)+\lambda(AP)\\
&=&(\frac{m+l}{m+n+2l}\delta(I)A+\frac{n+l}{m+n+2l}A\delta(I))P+\lambda(AP)\\
&=&(\delta(A)-\lambda(A))P+\lambda(AP)\\
&=&\delta(A)P+\lambda(AP)-\lambda(A)P.
\end{eqnarray*}

(iii) The proof is analogous to the proof of (ii).
\end{proof}

An subset $\mathcal{I}$ of an algebra $\mathcal{A}$ is called a \emph{left separating set} of $\mathcal{A}$ if for every $A\in \mathcal{A}$, $A\mathcal{I}={0}$ implies $A=0$. We have the following simple but noteworthy result.

\begin{corollary}\label{205}
Suppose $\mathcal{I}$ is a left separating left ideal of a unital algebra $\mathcal{A}$ and is contained in the algebra generated by all idempotents in $\mathcal{A}$. Then each weak \emph{(}\textit{m,n,l}\emph{)}-Jordan centralizer $\delta$ from $\mathcal{A}$ into itself is a centralizer.
\end{corollary}
\begin{proof}
Since $\mathcal{I}$ is contained in the algebra generated by all idempotents in $\mathcal{A}$ and by (i) of Lemma \ref{204}, we have that $\delta(I)\in \mathcal{I}'$, where $\mathcal{I}'$ denotes the commutant of $\mathcal{I}$. Hence $\delta(A)=\delta(I)A+\lambda(A)=A\delta(I)+\lambda(A)$ for every $A\in \mathcal{I}$ according to (\ref{2002}). For any $A(\neq \mathbb{K}I)\in\mathcal{I}$, we have
\begin{eqnarray*}
&&(m+n+l)(\delta(I)A^2+\lambda(A^2))\\
&=&(m+n+l)\delta(A^2)\\
&=&m\delta(A)A+nA\delta(A)+lA\delta(I)A+\lambda_AI\\
&=&m(\delta(I)A^2+\lambda(A)A)+n(A^2\delta(I)+A\lambda(A))+lA^2\delta(I)+\lambda_AI,
\end{eqnarray*}
which implies $\lambda(A)A=kI$ for some $k\in\mathbb{K}$. Hence $\lambda(A)=0$ and $\delta(A)=\delta(I)A=A\delta(I)$ for every $A\in \mathcal{I}$.
Then Lemma \ref{204} yields $A\delta(I)\mathcal{I}=A\mathcal{I}\delta(I)=\delta(A\mathcal{I})=\delta(I)A\mathcal{I}$, and since $\mathcal{I}$
is a separating left ideal, we have $A\delta(I)=\delta(I)A$ for every $A\in \mathcal{A}$.
Therefore,
$\delta(A)=\delta(I)A+\lambda(A)=A\delta(I)+\lambda(A)$ for every $A\in \mathcal{A}$. Now by the same argument as above, we have that $\delta(A)=\delta(I)A=A\delta(I)$ for every $A\in \mathcal{A}$ and this completes the proof.
\end{proof}

\begin{remark}
\emph{By \cite[Proposition 2.2]{Hadwin}, \cite[Example 6.2]{Samei}, we see that the class of algebras we discussed in Corollary \ref{205} contains a lot of algebras and is therefore very large.}

\end{remark}

The proof of the following lemma is similar to the proof of \cite[Proposition 1.1]{LDRA} and we omit it.

\begin{lemma}\label{206}
Let $E$ and $F$ be non-zero subspaces of $X$ and $X^*$ respectively.
Let $\phi: E\times F\rightarrow B(X)$ be a bilinear mapping such that
$\phi(x,f)X\subseteq \mathbb{K}x$ for all $x\in E$ and $f\in F$.
Then there exists a linear mapping $S: F\rightarrow X^*$ such that
$\phi(x,f)=x\otimes Sf$ for all $x\in E$ and $f\in F$.
\end{lemma}

\begin{lemma}\label{207}
Let $\mathcal{L}$ be a subspace lattice on a Banach space X
and $\delta$ be a weak \emph{(}\textit{m,n,l}\emph{)}-Jordan centralizer from $alg\mathcal{L}$ into itself.
Suppose that $E$ and $L$ are in $\mathcal{J}(\mathcal{L})$ such that $E_{-}\ngeq L$.
Let $x$ be in E and $f$ be in $L_-^\perp$. Then $(\delta(x\otimes f)-\lambda(x\otimes f))X\subseteq \mathbb{K}x$.
\end{lemma}

\begin{proof}
Since $E_{-}\ngeq L$, we have that $E\leq L$. So $x\otimes f\in \textrm{alg}\mathcal{L}$. Suppose $f(x)\neq0$,
it follows from Lemmas \ref{201} and \ref{204} that $\lambda(x\otimes f)=0$ and $\delta(x\otimes f)=x\otimes f\delta(I)$. Thus $\delta(x\otimes f)X\subseteq \mathbb{K}x$.
\

Now we assume $f(x)=0$. Choose $z$ from $L$ and $g$ from $E_-^\perp$ such that $g(z)=1$. Then
\begin{eqnarray*}
&&(m+n+2l)(m+n+l)\delta(x\otimes f)\\
&&~~~~~~~=(m+n+2l)(m+n+l)\delta((x\otimes g)(z\otimes f)+(z\otimes f)(x\otimes g))\\
&&~~~~~~~=(m+n+2l)(m\delta(x\otimes g)(z\otimes f)+n(x\otimes g)\delta(z\otimes f)+l(x\otimes g)\delta(I)(z\otimes f))\\
&&~~~~~~~~~~+(m+n+2l)(m\delta(z\otimes f)(x\otimes g)+n(z\otimes f)\delta(x\otimes g)+l(z\otimes f)\delta(I)(x\otimes g))\\
&&~~~~~~~~~~+(m+n+2l)(\lambda_{x\otimes g+z\otimes f}-\lambda_{x\otimes g}-\lambda_{z\otimes f})I\\
&&~~~~~~~=(m^2+ml)\delta(I)x\otimes f+(n^2+nl)x\otimes f\delta(I)\\
&&~~~~~~~~~~+2(mn+ml+nl+l^2)(x\otimes g\delta(I)z\otimes f+z\otimes f\delta(I)x\otimes g)+\lambda_1I
\end{eqnarray*}
for some $\lambda_1\in \mathbb{K}$.

On the other hand,
\begin{eqnarray*}
&&(m+2n+l)(m+n+l)\delta(x\otimes f)\\
&&~~~=(m+n+l)((m+l)\delta(I)x\otimes f+(n+l)x\otimes f\delta(I)+(\lambda_{x \otimes f+I}-\lambda_{x\otimes f})I)\\
&&~~~=(m^2+2ml+l^2+mn+nl)\delta(I)x\otimes f+(ml+mn+l^2+2nl+n^2)x\otimes f\delta(I)+\lambda_2I
\end{eqnarray*}
for some $\lambda_2 \in \mathbb{K}$.

So
\begin{eqnarray}
\delta(I)x\otimes f+x\otimes f\delta(I)=2x\otimes g\delta(I)z\otimes f+2z\otimes f\delta(I)x\otimes g+ \lambda I \label{2004}
\end{eqnarray}
for some $\lambda\in \mathbb{K}$.

Notice that this equation is valid for all $z$ in $L$ satisfying $g(z)=1$. Applying (\ref{2004}) to $x$, we have
\begin{eqnarray}
f(\delta(I)x)x=2g(x)f(\delta(I)x)z+\lambda x.\label{2005}
\end{eqnarray}

If $g(x)=0$ and $f(z)=0$, then $f(\delta(I)x)=\lambda$. Substituting $z+x$ for $z$ in (\ref{2004}) gives
\begin{eqnarray}
\delta(I)x\otimes f+x\otimes f\delta(I)=2x\otimes g\delta(I)(z+x)\otimes f+2\lambda(z+x)\otimes g+\lambda I\label{2011}.
\end{eqnarray}
Comparing (\ref{2004}) with (\ref{2011}) yields
\begin{eqnarray}
g(\delta(I)x)x\otimes f+\lambda x\otimes g=0.
\end{eqnarray}
Applying this equation to $z$ yields $\lambda x=0,$ which means $f(\delta(I)x)=\lambda=0$.

If $g(x)=0$ and $f(z)\neq0$, we also have $f(\delta(I)x)=\lambda$, and it follows from Lemma \ref{204} that
\begin{eqnarray*}
\delta(I)x\otimes f+x\otimes f\delta(I)&=&2x\otimes g\delta(I)z\otimes f+2z\otimes f\delta(I)x\otimes g+\lambda I\\
&=&2(x\otimes g)(z\otimes f)\delta(I)+2\delta(I)(z\otimes f)(x\otimes g)+\lambda I\\
&=&2x\otimes f\delta(I)+\lambda I,
\end{eqnarray*}
whence $$\delta(I)x\otimes f=x\otimes f\delta(I)+\lambda I.$$
Applying the above equation to $x$ leads to $f(\delta(I)x)=-\lambda$. Hence $f(\delta(I)x)=\lambda=0$.

If $g(x)\neq 0$, replacing $z$ by $\frac{1}{g(x)}x$ in (\ref{2005}) gives $f(\delta(I)x)=-\lambda$,
while
\begin{eqnarray}
\delta(I)x\otimes f+x\otimes f\delta(I)&=&2x\otimes g\delta(I)z\otimes f+2 z\otimes f\delta(I)x\otimes g+\lambda I\nonumber\\
&=&2\delta(I)(x\otimes g) (z\otimes f)+2 (z\otimes f) (x\otimes g)\delta(I)+\lambda I\nonumber\\
&=&2\delta(I)(x\otimes f)+\lambda I.\nonumber
\end{eqnarray}
Hence
\begin{eqnarray}
x\otimes f\delta(I)=\delta(I)x\otimes f+\lambda I.\label{2013}
\end{eqnarray}
Applying (\ref{2013}) to $x$ yields $f(\delta(I)x)=\lambda$. Therefore, $f(\delta(I)x)=\lambda=0$.

So by (\ref{2004}), we obtain $\delta(I)x\otimes f=2g(\delta(I)z)x\otimes f-x\otimes f\delta(I)$.
It follows from Lemma \ref{201} that
\begin{eqnarray*}
\delta(x\otimes f)&=&\frac{m+l}{m+n+2l}\delta(I)(x\otimes f)+\frac{n+l}{m+n+2l}(x\otimes f)\delta(I)+\lambda(x\otimes f)\\
&=&\frac{m+l}{m+n+2l}(2g(\delta(I)z)x\otimes f-x\otimes f\delta(I))+\frac{n+l}{m+n+2l}(x\otimes f)\delta(I)+\lambda(x\otimes f)\\
&=&\frac{2(m+l)}{m+n+2l}g(\delta(I)z)x\otimes f+\frac{n-m}{m+n+2l}(x\otimes f)\delta(I)+\lambda(x\otimes f).
\end{eqnarray*}
Hence $(\delta(x\otimes f)-\lambda(x\otimes f))X\subseteq \mathbb{K}x$.
\end{proof}

\begin{theorem}\label{208}
Let $\mathcal{L}$ be a subspace lattice on a Banach space $X$ satisfying $\vee\{F: F\in \mathcal{J}(\mathcal{L})\}=X$.
If $\delta$ is a weak $(m,n,l)$-Jordan centralizer from $alg\mathcal{L}$ into itself, then $\delta$ is a centralizer.
In particular, the conclusion holds if $\mathcal{L}$ has the property $X_-\neq X$.
\end{theorem}

\begin{proof}
Let $E$ be in $\mathcal{J}(\mathcal{L})$. By $\vee\{F: F\in \mathcal{J}(\mathcal{L})\}=X$,
there is an element $L$ in $\mathcal{J}(\mathcal{L})$ such that $E_-\ngeq L$.
Let $x$ be in $E$ and $f$ be in $(L_-)^\bot$. Let $\overline{\delta}=\delta-\lambda$. Then $\overline{\delta}(I)=\delta(I)$, and it follows from Lemmas \ref{206} and \ref{207} that there exists a linear mapping
$S: (L_-)^\bot\rightarrow X^*$ such that
$$\overline{\delta}(x\otimes f)=x\otimes Sf.$$
This together with $$\frac{m+l}{m+n+2l}\overline{\delta}(I)x\otimes f+\frac{n+l}{m+n+2l}x\otimes f\overline{\delta}(I)=\overline{\delta}(x\otimes f)$$ leads to
$$x\otimes (Sf-\frac{n+l}{m+n+2l}\overline{\delta}(I)^*f)=\frac{m+l}{m+n+2l}\overline{\delta}(I)x\otimes f.$$
Thus there exists a constant $\lambda_E$ in $\mathbb{K}$ such that $\overline{\delta}(I)x=\lambda_Ex$ for every $x\in E$.
Similarly, for every $y\in L$, we have $\overline{\delta}(I)y=\lambda_Ly$.

If $f(x)\neq 0$, then $\overline{\delta}(x\otimes f)=\overline{\delta}(I)x\otimes f=x\otimes f\overline{\delta}(I)$.

If $f(x)=0$, according to the proof of Lemma \ref{207}, we can choose $z$ from $L$ and $g$ from $E_-^\perp$ such that $g(z)=1$ and
$\overline{\delta}(I)x\otimes f=2g(\overline{\delta}(I)z)x\otimes f-x\otimes f\overline{\delta}(I)$.
Since $x\in E\leq L$, we have $\overline{\delta}(I)x=\lambda_Lx$. Thus
$$\overline{\delta}(I)x\otimes f=2\lambda_Lx\otimes f-x\otimes f\overline{\delta}(I)=2\overline{\delta}(I)x\otimes f-x\otimes f\overline{\delta}(I).$$
Hence $\overline{\delta}(x\otimes f)=\overline{\delta}(I)x\otimes f=x\otimes f\overline{\delta}(I)$.\

Then for any $x\in E$, $f\in (L_-)^\bot$ and $A\in \textrm{alg}\mathcal{L}$, we have
\begin{eqnarray*}
A\overline{\delta}(I)x\otimes f=Ax\otimes f\overline{\delta}(I)=\overline{\delta}(I)Ax\otimes f.
\end{eqnarray*}
So we have $A\overline{\delta}(I)x=\overline{\delta}(I)Ax$ for any $x\in E$. By $\vee\{F: F\in \mathcal{J}(\mathcal{L})\}=X$,
we have $\overline{\delta}(A)=A\overline{\delta}(I)=\overline{\delta}(I)A$ for any $A\in \textrm{alg}\mathcal{L}$, this means $\delta(A)=A\delta(I)+\lambda(A)=\delta(I)A+\lambda(A)$. The remaining part goes along the same line as the proof of Corollary \ref{205} and this completes the proof.
\end{proof}

\begin{remark}
\emph{By \cite{SRL}, a subspace lattice $\mathcal{L}$ is said to be \textit{completely~distributive}
if $L=\vee\{E\in \mathcal{L}: E_-\ngeq L\}$ or $L=\wedge\{E_-: E\in \mathcal{L}~\mathrm{and}~E\nleq L\}$ for all $L\in \mathcal{L}$.
It follows that completely distributive subspace lattices satisfy the condition
$\vee \{E: E\in \mathcal{J}(\mathcal{L}) \}=X$. Thus Theorem \ref{208} applies to completely distributive subspace lattice algebras.
A subspace lattice $\mathcal{L}$ is called a \textit{$\mathcal{J}$-subspace lattice} on $X$ if
$\vee\{K: K\in \mathcal{J}(\mathcal{L})\}=X$,
$\wedge\{K_-: K\in \mathcal{J}(\mathcal{L})\}=(0)$,
$K\vee K_-=X$ and
$K\wedge K_-=(0)~\mathrm{for~any}~K\in \mathcal{J}(\mathcal{L})$. Note also that the condition
$\vee \{K: K\in \mathcal{J}(\mathcal{L}) \}=X$ is part of the definition of $\mathcal{J}$-subspace lattices, thus Theorem \ref{208} also
applies to $\mathcal{J}$-subspace lattice algebras.}
\end{remark}

With a proof similar to the proof of Theorem \ref{208}, we have the following theorem.
\begin{theorem}\label{209}
Let $\mathcal{L}$ be a subspace lattice on a Banach space $X$
satisfying $\wedge\{L_-: L\in \mathcal{J}(\mathcal{L})\}=(0)$.
If $\delta$ is a weak $(m,n,l)$-Jordan centralizer from $alg\mathcal{L}$ into itself, then $\delta$ is a centralizer.
In particular, the conclusion holds if $\mathcal{L}$ has the property $(0)_+\neq (0)$.
\end{theorem}
As for the cases of $(m,n,l)$-Jordan centralizers, we have from Remark \ref{211} , Theorem \ref{208} and Theorem \ref{209} the following theorem.
\begin{theorem}
Let $\mathcal{L}$ be a subspace lattice on a Banach space $X$
satisfying $\vee\{F: F\in \mathcal{J}(\mathcal{L})\}=X$ or $\wedge\{L_-: L\in \mathcal{J}(\mathcal{L})\}=(0)$. If $\delta$ is an $(m,n,l)$-Jordan centralizer from $alg\mathcal{L}$ to $B(X)$, then $\delta$ is a centralizer.
\end{theorem}

In the rest of this section we will investigate weak $(m,n,l)$-Jordan centralizer on CSL algebras.
Let $H$ be a complex separable Hilbert space and $\mathcal{L}$ be a CSL on $H$.
Let $\mathcal{L^\perp}$ be the lattice \{$I-E: E\in \mathcal{L}$\} and $\mathcal{L^\prime}$ be the commutant of $\mathcal{L}$.
It is easy to verify that $(\textrm{alg}\mathcal{L})^*=\textrm{alg}\mathcal{L^\perp}$ for any lattice $\mathcal{L}$ on $H$ and the diagonal
$(\textrm{alg}\mathcal{L})\cap(\textrm{alg}\mathcal{L})^*=\mathcal{L^\prime}$ is a von
Neumann algebra. Given a CSL $\mathcal{L}$ on a Hilbert space $H$,
we define $G_1(\mathcal{L})$ and $G_2(\mathcal{L})$ to be the
projections onto the closures of the linear spans of $\{EA(I-E)x:
E\in \mathcal{L}, A\in alg\mathcal{L}, x\in H\}$ and $\{(I-E)A^*Ex:
E\in \mathcal{L}, A\in alg\mathcal{L}, x\in H\}$, respectively. For simplicity, we
write $G_1$ and $G_2$ for $G_1(\mathcal{L})$ and $G_2(\mathcal{L})$.
Since CSL is reflexive, it is easy to verify that $G_1\in
\mathcal{L}$ and $G_2\in \mathcal{L}^\bot$. In \cite{JDCSL}, Lu
showed that $G_1\vee G_2\in \mathcal{L}\cap\mathcal{L}^\bot$ and
$\textrm{alg}\mathcal{L}(I-G_1\vee G_2)\subseteq \mathcal{L}^\prime$.

\begin{theorem}\label{210}
Let $\mathcal{L}$ be a CSL on $H$. If $\delta$ is a bounded weak $(m,n,l)$-Jordan centralizer from $alg\mathcal{L}$ into itself,
then $\delta$ is a centralizer.
\end{theorem}
\begin{proof}
We divide the proof into two cases.\\
Case 1: Suppose $G_1\vee G_2=I$.\

Let $A\in \textrm{alg}\mathcal{L}$. For any $T\in \textrm{alg}\mathcal{L}$ and $P\in\mathcal{L}$,
since $PT(I-P)=P-(P-PT(I-P))$ is a difference of two idempotents, it follows from Lemma \ref{204} that
\begin{eqnarray*}
\delta(I)APT(I-P)=A\delta(I)PT(I-P)=\delta(APT(I-P))=\delta(A)PT(I-P)-\lambda(A)PT(I-P).
\end{eqnarray*}

By arbitrariness of $P$ and $T$, we have $A\delta(I)G_1=\delta(I)AG_1=(\delta(A)-\lambda(A))G_1$. That is,
\begin{eqnarray*}
\delta(A)G_1=(A\delta(I)+\lambda(A))G_1=(\delta(I)A+\lambda(A))G_1,
\end{eqnarray*}
whence
\begin{eqnarray}
\delta(AG_1)=\delta(A)G_1+\lambda(AG_1)-\lambda(A)G_1=\delta(I)AG_1+\lambda(AG_1)=A\delta(I)G_1+\lambda(AG_1)\label{3001}.
\end{eqnarray}

Define $\delta^*(A^*)=\delta(A)^*$ for every $A^*\in \textrm{alg}\mathcal{L}^\bot$. So
\begin{eqnarray*}
(m+n+l)\delta^*((A^*)^2)&=&((m+n+l)\delta(A^2))^*\\
&=&(m\delta(A)A+nA\delta(A)+lA\delta(I)A+\lambda_AI)^*\\
&=&mA^*\delta^*(A^*)+n\delta^*(A^*)A^*+lA^*\delta^*(I)A^*+\lambda_{A^*},
\end{eqnarray*}
where $\lambda_{A^*}=\overline{\lambda_A}$.

With the proof similar to the proof of (\ref{3001}), we have $$G_2\delta(I)A=G_2A\delta(I)=G_2(\delta(A)-\lambda(A)).$$
So by $G_1\vee G_2=I$,
\begin{eqnarray*}
(I-G_1)\delta(I)A=(I-G_1)A\delta(I)=(I-G_1)(\delta(A)-\lambda(A)),
\end{eqnarray*}
whence
\begin{eqnarray}
\delta((I-G_1)A)&=&(1-G_1)\delta(A)+\lambda((I-G_1)A)-\lambda(A)(I-G_1)\nonumber\\
&=&(1-G_1)(\delta(A)-\lambda(A))+\lambda((I-G_1)A)\nonumber\\
&=&(1-G_1)\delta(I)A+\lambda((I-G_1)A)\nonumber\\
&=&(I-G_1)A\delta(I)+\lambda((I-G_1)A)\label{3004}.
\end{eqnarray}

Hence by (\ref{3001}) and (\ref{3004}),
\begin{eqnarray*}
\delta(A)&=&\delta(AG_1+G_1A(I-G_1)+(I-G_1)A)\\
&=&A\delta(I)G_1+\lambda(AG_1)+G_1A(I-G_1)\delta(I)+(I-G_1)A\delta(I)+\lambda((1-G_1)A)\\
&=&G_1A\delta(I)G_1+G_1A\delta(I)(I-G_1)+(I-G_1)A\delta(I)\\
&&+\lambda(AG_1)+\lambda((1-G_1)A)+\lambda(G_1A(1-G_1))\\
&=&A\delta(I)+\lambda(A),
\end{eqnarray*}
and similarly, $\delta(A)=\delta(I)A+\lambda(A)$. The remaining part goes along the same line as the proof of Corollary \ref{205} and we conclude that $\delta$ is a centralizer in this case.\\
Case 2: Suppose $G_1\vee G_2<I$.

Let $G=G_1\vee G_2$. Since $G\in \mathcal{L}\cap\mathcal{L}^\bot$ and $\textrm{alg}\mathcal{L}(I-G)\subseteq\mathcal{L}^\prime$,
so $(I-G)\textrm{alg}\mathcal{L}(I-G)$ is a von Neumann algebra. The algebra $\textrm{alg}\mathcal{L}$ can be written as the direct sum
$$\textrm{alg}\mathcal{L}=\textrm{alg}(G\mathcal{L}G)\oplus \textrm{alg}((I-G)\mathcal{L}(I-G)).$$
By Lemma \ref{204} we have that $$\delta(GAG)=G\delta(A)G~\mathrm{and}~\delta((I-G)A(I-G))=(I-G)\delta(A)(I-G)$$
for every $A\in \textrm{alg}\mathcal{L}$. Therefore $\delta$ can be written as $\delta^{(1)}\oplus\delta^{(2)}$,
where $\delta^{(1)}$ is a weak (\textit{m,n,l})-Jordan centralizer from $\textrm{alg}(G\mathcal{L}G)$ into itself
and $\delta^{(2)}$ is a weak (\textit{m,n,l})-Jordan centralizer from $\textrm{alg}((I-G)\mathcal{L}(I-G))$ into itself.
It is easy to show that $G_1(G\mathcal{L}G)\vee G_2(G\mathcal{L}G)=G$.
So it follows from Case 1 that $\delta^{(1)}$ is a centralizer on $\textrm{alg}(G\mathcal{L}G)$.
$(I-G)\textrm{alg}\mathcal{L}(I-G)$ is a von Neumann algebra and $\delta^{(2)}$ is continuous, so by Corollary \ref{205},
$\delta^{(2)}$ is a centralizer on $\textrm{alg}((I-G)\mathcal{L}(I-G))$. Consequently, $\delta$ is a centralizer on $\textrm{alg}\mathcal{L}$.
\end{proof}

\section{Centralizers of generalized matrix algebras}

\

We call $\mathcal M$ a \textsl{unital $\mathcal A$-bimodule} if $\mathcal M$ is an
$\mathcal A$-bimodule and satisfies $I_{\mathcal A}M=M I_{\mathcal A}=M$ for every $M\in{\mathcal M}$.
We call $\mathcal{M}$ a \textit{faithful left $\mathcal{A}$-module} if for any $A\in \mathcal{A}$,
$A\mathcal{M}=0$ implies $A=0$. Similarly, we can define a \textit{faithful right $\mathcal{B}$-module}. \

Throughout this section, we denote the generalized matrix algebra originated from the Morita context
$(\mathcal{A}, \mathcal{B}, \mathcal{M}, \mathcal{N}, \phi_{\mathcal{M}\mathcal{N}}, \varphi_{\mathcal{N}\mathcal{M}})$
by
$\mathcal{U}=\left[\begin{array}{cc}\mathcal{A} & \mathcal{M} \\ \mathcal{N} & \mathcal{B} \\ \end{array} \right]$
,
where $\mathcal{A}, \mathcal{B}$ are two unital algebras over a number field $\mathbb{K}$ and $\mathcal{M}, \mathcal{N}$ are two unital bimodules, and
at least one of $\mathcal{M}$ and $\mathcal{N}$ is distinct from zero.
We use the symbols $I_{\mathcal A}$ and $I_\mathcal{B}$ to denote the unit element in $\mathcal A$ and $\mathcal{B}$, respectively.
Moreover, we make no difference between $\lambda(A)=\frac{1}{m+n+2l}(\lambda_{A+I}-\lambda_A)I$ and $\frac{1}{m+n+2l}(\lambda_{A+I}-\lambda_A)\in\mathbb{K}$.
\begin{lemma}\label{30101}
Let $\delta$ be a weak $(m,n,l)$-Jordan centralizer from $\mathcal{U}$ into itself. Then $\delta$ is of the form
$$\delta\left(\left[\begin{array}{cc}  A   & M \\  N  & B  \\ \end{array}\right]\right)=
\left[\begin{array}{cc}  a_{11}(A)+\lambda\left(\left[
                                            \begin{array}{cc}
                                              0 & M \\
                                              N & B \\
                                            \end{array}
                                          \right]\right)I_\mathcal{A}
 &c_{12}(M) \\ d_{21}(N) &  b_{22}(B)+\lambda\left(\left[
                                               \begin{array}{cc}
                                                 A & M \\
                                                 N & 0 \\
                                               \end{array}
                                             \right]\right)I_\mathcal{B}
   \\   \end{array}\right],$$
for any $A\in \mathcal{A}$, $M\in \mathcal{M}$, $N\in \mathcal{N}$, $B\in \mathcal{B}$,
where $a_{11}: \mathcal{A}\rightarrow \mathcal{A}$, $c_{12}: \mathcal{M}\rightarrow \mathcal{M}$,
$d_{21}: \mathcal{N}\rightarrow \mathcal{N}$, $b_{22}: \mathcal{B}\rightarrow \mathcal{B}$
are all linear mappings satisfying
$$c_{12}(M)=a_{11}(I_\mathcal{A})M=Mb_{22}(I_\mathcal{B})~~~and~~~d_{21}(N)=Na_{11}(I_\mathcal{A})=b_{22}(I_\mathcal{B})N.$$
\end{lemma}

\begin{proof}
Assume that $\delta$ is a weak $(m,n,l)$-Jordan centralizer from $\mathcal{U}$ into itself.
Because $\delta$ is linear, for any $A\in \mathcal{A}$, $M\in \mathcal{M}$, $N\in \mathcal{N}$, $B\in \mathcal{B}$,
we can write
 \begin{eqnarray*}
&&\delta\left(\left[\begin{array}{cc}A & M \\N & B \\\end{array}\right]\right)\\
&=&\left[\begin{array}{cc}a_{11}(A)+b_{11}(B)+c_{11}(M)+d_{11}(N) & a_{12}(A)+b_{12}(B)+c_{12}(M)+d_{12}(N) \\a_{21}(A)+b_{21}(B)+c_{21}(M)+d_{21}(N) & a_{22}(A)+b_{22}(B)+c_{22}(M)+d_{22}(N) \\\end{array}\right],
\end{eqnarray*}
where
$a_{ij}$, $b_{ij}$, $c_{ij}$, $d_{ij}$ are linear mappings, $i, j\in\{1, 2\}$.

Let $P=\left[\begin{array}{cc}I_\mathcal{A} & 0 \\0 & 0 \\\end{array}\right]$ and
for any $A\in \mathcal{A}$, $S=\left[\begin{array}{cc} A & 0 \\  0 & 0  \end{array}\right]$.
By Lemma \ref{204}, $\delta(PS)=P\delta(S)+\lambda(PS)-\lambda(S)P$ and $\delta(SP)=\delta(S)P+\lambda(SP)-\lambda(S)P$, so we have
\begin{eqnarray*}
&&\left[\begin{array}{cc}a_{11}(A) & a_{12}(A) \\a_{21}(A) & a_{22}(A)\\\end{array}\right]
=\delta\left(\left[\begin{array}{cc} A & 0 \\  0 & 0  \end{array}\right]\right)
=\delta\left(\left[\begin{array}{cc}I_\mathcal{A} & 0 \\0 & 0 \\\end{array}\right]
\left[\begin{array}{cc} A & 0 \\  0 & 0  \end{array}\right]\right)\\
&=&\left[\begin{array}{cc}I_\mathcal{A} & 0 \\0 & 0 \\\end{array}\right]
\delta\left(\left[\begin{array}{cc} A & 0 \\  0 & 0  \end{array}\right]\right)+\left[
                                                                                 \begin{array}{cc}
                                                                                   \lambda(PS)I_\mathcal{A} & 0 \\
                                                                                   0 & \lambda(PS)I_\mathcal{B} \\
                                                                                 \end{array}
                                                                               \right]-\left[
                                                                                         \begin{array}{cc}
                                                                                           \lambda(S)I_\mathcal{A}& 0 \\
                                                                                           0 & 0 \\
                                                                                         \end{array}
                                                                                       \right]\\
&=&\left[\begin{array}{cc}a_{11}(A) & a_{12}(A) \\0 & \lambda\left(\left[
                                                                \begin{array}{cc}
                                                                  A & 0 \\
                                                                  0 & 0 \\
                                                                \end{array}
                                                              \right]\right)I_\mathcal{B}
\\\end{array}\right]
\end{eqnarray*}
and
\begin{eqnarray*}
&&\left[\begin{array}{cc}a_{11}(A) & a_{12}(A) \\a_{21}(A) & a_{22}(A)\\\end{array}\right]
=\delta\left(\left[\begin{array}{cc} A & 0 \\  0 & 0  \end{array}\right]\right)
=\delta\left(\left[\begin{array}{cc} A & 0 \\  0 & 0  \end{array}\right]
\left[\begin{array}{cc}I_\mathcal{A} & 0 \\0 & 0 \\\end{array}\right]\right)\\
&=&\delta\left(\left[\begin{array}{cc} A & 0 \\  0 & 0  \end{array}\right]\right)
\left[\begin{array}{cc}I_\mathcal{A} & 0 \\0 & 0 \\\end{array}\right]+\left[
                                                                                 \begin{array}{cc}
                                                                                   \lambda(SP)I_\mathcal{A} & 0 \\
                                                                                   0 & \lambda(SP)I_\mathcal{B} \\
                                                                                 \end{array}
                                                                               \right]-\left[
                                                                                         \begin{array}{cc}
                                                                                           \lambda(S)I_\mathcal{A}& 0 \\
                                                                                           0 & 0 \\
                                                                                         \end{array}
                                                                                       \right]\\
&=&\left[\begin{array}{cc}a_{11}(A) &  0\\a_{21}(A) & \lambda\left(\left[
                                                                \begin{array}{cc}
                                                                  A & 0 \\
                                                                  0 & 0 \\
                                                                \end{array}
                                                              \right]\right)I_\mathcal{B}\\\end{array}\right].
\end{eqnarray*}
So we have $a_{12}(A)=0$, $a_{21}(A)=0$ and $a_{22}(A)=\lambda\left(\left[
                                                                \begin{array}{cc}
                                                                  A & 0 \\
                                                                  0 & 0 \\
                                                                \end{array}
                                                              \right]\right)I_\mathcal{B}$.
\

Similarly, by considering $S=\left[\begin{array}{cc}0 & M \\0 & 0 \\\end{array}\right]$ and
$P=\left[\begin{array}{cc}I_\mathcal{A} & 0 \\0 & 0 \\\end{array}\right]$, we obtain that
$c_{11}(M)=\lambda\left(\left[
                                                                \begin{array}{cc}
                                                                  0 & M \\
                                                                  0 & 0 \\
                                                                \end{array}
                                                              \right]\right)I_\mathcal{A}$, $c_{21}(M)=0$ and $c_{22}(M)=\lambda\left(\left[
                                                                \begin{array}{cc}
                                                                  0 & M \\
                                                                  0 & 0 \\
                                                                \end{array}
                                                              \right]\right)I_\mathcal{B}$ for every $M\in \mathcal{M}$.\

By considering $S=\left[\begin{array}{cc}0 & 0 \\N & 0 \\\end{array}\right]$ and
$P=\left[\begin{array}{cc}I_\mathcal{A} & 0 \\0 & 0 \\\end{array}\right]$, we obtain
$d_{11}(N)=\lambda\left(\left[
                                                                \begin{array}{cc}
                                                                  0 & 0 \\
                                                                  N & 0 \\
                                                                \end{array}
                                                              \right]\right)I_\mathcal{A}$, $d_{12}(N)=0$ and $d_{22}(N)=\lambda\left(\left[
                                                                \begin{array}{cc}
                                                                  0 & 0 \\
                                                                  N & 0 \\
                                                                \end{array}
                                                              \right]\right)I_\mathcal{B}$ for every $N\in \mathcal{N}$.\

By considering $S=\left[\begin{array}{cc}0 & 0 \\0 & B \\\end{array}\right]$ and
$Q=\left[\begin{array}{cc}0 & 0 \\0 & I_\mathcal{B} \\\end{array}\right]$, we obtain
$b_{11}(B)=\lambda\left(\left[
                                                                \begin{array}{cc}
                                                                  0 & 0 \\
                                                                  0 & B \\
                                                                \end{array}
                                                              \right]\right)I_\mathcal{A}$, $b_{12}(B)=0$ and $b_{21}(B)=0$ for every $B\in \mathcal{B}$.\

For any $A\in \mathcal{A}$, $M_1\in \mathcal{M}$, $M_2\in \mathcal{M}$ and $B\in \mathcal{B}$,
let $S=\left[\begin{array}{cc}A & M_1 \\0 & 0 \\\end{array}\right]$ and $T=\left[\begin{array}{cc}0 & M_2 \\0 & B \\\end{array}\right]$.

Then
\begin{eqnarray*}
&&(m+n+l)\left[\begin{array}{cc}0 & c_{12}(AM_2+M_1B) \\0 & 0 \\\end{array}\right]\\
&=&(m+n+l)\delta(ST)=(m+n+l)\delta(ST+TS)\\
&=&m\left[\begin{array}{cc}a_{11}(A)+\lambda\left(\left[
                                                                \begin{array}{cc}
                                                                  0 & M_1 \\
                                                                  0 & 0 \\
                                                                \end{array}
                                                              \right]\right)I_\mathcal{A} & c_{12}(M_1) \\0 & \lambda\left(\left[
                                                                \begin{array}{cc}
                                                                  A & M_1 \\
                                                                  0 & 0 \\
                                                                \end{array}
                                                              \right]\right)I_\mathcal{B} \\\end{array}\right] \left[\begin{array}{cc}0 & M_2 \\0 & B \\\end{array}\right]\\
&&+ m\left[\begin{array}{cc}\lambda\left(\left[
                                                                \begin{array}{cc}
                                                                  0 & M_2 \\
                                                                  0 & B \\
                                                                \end{array}
                                                              \right]\right)I_\mathcal{A} & c_{12}(M_2) \\0 & b_{22}(B)+\lambda\left(\left[
                                                                \begin{array}{cc}
                                                                  0 & M_2 \\
                                                                  0 & 0 \\
                                                                \end{array}
                                                              \right]\right)I_\mathcal{B} \\\end{array}\right] \left[\begin{array}{cc}A & M_1 \\0 & 0 \\\end{array}\right]\\
&&+n\left[\begin{array}{cc}A & M_1 \\0 & 0 \\\end{array}\right]
\left[\begin{array}{cc}\lambda\left(\left[
                                                                \begin{array}{cc}
                                                                  0 & M_2 \\
                                                                  0 & B \\
                                                                \end{array}
                                                              \right]\right)I_\mathcal{A} & c_{12}(M_2) \\0 & \lambda\left(\left[
                                                                \begin{array}{cc}
                                                                  0 & M_2 \\
                                                                  0 & 0 \\
                                                                \end{array}
                                                              \right]\right)I_\mathcal{B}+b_{22}(B) \\\end{array}\right]
\end{eqnarray*}
\begin{eqnarray*}
&&+n\left[\begin{array}{cc}0 & M_2 \\0 & B \\\end{array}\right]
\left[\begin{array}{cc}a_{11}(A)+\lambda\left(\left[
                                                                \begin{array}{cc}
                                                                  0 & M_1 \\
                                                                  0 & 0 \\
                                                                \end{array}
                                                              \right]\right)I_\mathcal{A} & c_{12}(M_1) \\0 & \lambda\left(\left[
                                                                \begin{array}{cc}
                                                                  A & M_1 \\
                                                                  0 & 0 \\
                                                                \end{array}
                                                             \right]\right)I_\mathcal{B}\\\end{array}\right]\\
&&+ l\left[\begin{array}{cc}A & M_1 \\0 & 0 \\\end{array}\right] \left[\begin{array}{cc}a_{11}(I_\mathcal{A}) & 0 \\0 & b_{22}(I_\mathcal{B}) \\\end{array}\right]
\left[\begin{array}{cc}0 & M_2 \\0 & B \\\end{array}\right]\\
&&+ \left[\begin{array}{cc}(\lambda_{S+T}-\lambda_S-\lambda_T) I_\mathcal{A} & 0 \\0 & (\lambda_{S+T}-\lambda_S-\lambda_T) I_\mathcal{B} \\\end{array}\right].
\end{eqnarray*}
The above matrix equation implies
\begin{eqnarray}
&&(m+n+l)c_{12}(AM_2+M_1B)\nonumber\\
&=&ma_{11}(A)M_2+m\lambda\left(\left[
                                       \begin{array}{cc}
                                         0 & M_1 \\
                                         0 & 0 \\
                                       \end{array}
                                     \right]\right)M_2
+ mc_{12}(M_1)B+m\lambda\left(\left[
                   \begin{array}{cc}
                     0 & M_2 \\
                     0 & B \\
                   \end{array}
                 \right]\right)M_1\nonumber\\
&&+nAc_{12}(M_2)+n\lambda\left(\left[
                        \begin{array}{cc}
                          0 & M_2 \\
                          0 & 0 \\
                        \end{array}
                      \right]\right)M_1
+nM_1b_{22}(B)+n\lambda\left(\left[
                         \begin{array}{cc}
                           A & M_1 \\
                           0 & 0 \\
                         \end{array}
                       \right]\right)M_2\nonumber\\
&&+lAa_{11}(I_\mathcal{A})M_2+lM_1b_{22}(I_\mathcal{B})B.\label{3000001}
\end{eqnarray}

Taking $B=0$ , $A=I_\mathcal{A}$ and $M_1=0$ in (\ref{3000001}),
we have $c_{12}(M)=a_{11}(I_\mathcal{A})M$ for every $M\in \mathcal{M}$.
Taking $A=0$, $B=I_\mathcal{B}$ and $M_2=0$ in (\ref{3000001}), we have $c_{12}(M)=Mb_{22}(I_\mathcal{B})$ for every $M\in \mathcal{M}$.

Symmetrically, $d_{21}(N)=b_{22}(I_\mathcal{B})N=Na_{11}(I_\mathcal{A})$ for every $N\in \mathcal{N}$.
\end{proof}

\begin{theorem}\label{30102}
Let $\delta$ be a weak $(m,n,l)$-Jordan centralizer from $\mathcal{U}$ into itself.
If one of the following conditions holds:\\
\emph{(1)}~~${\mathcal M}$ is a faithful left $\mathcal{A}$-module and a faithful right ${\mathcal B}$-module;\\
\emph{(2)}~~${\mathcal M}$ is a faithful left $\mathcal{A}$-module and ${\mathcal N}$ is a faithful left ${\mathcal B}$-module;\\
\emph{(3)}~~${\mathcal N}$ is a faithful right ${\mathcal A}$-module and a faithful left ${\mathcal B}$-module;\\
\emph{(4)}~~${\mathcal N}$ is a faithful right ${\mathcal A}$-module and ${\mathcal M}$ is a faithful right ${\mathcal B}$-module.\\
Then $\delta$ is a centralizer.
\end{theorem}

\begin{proof}
Let $\delta$ be an $(m,n,l)$-Jordan centralizer from $\mathcal{U}$ into itself. By Lemma \ref{30101}, we have
\begin{eqnarray}
c_{12}(M)=a_{11}(I_\mathcal{A})M=Mb_{22}(I_\mathcal{B})\label{30102}
\end{eqnarray}
for every $M\in \mathcal{M}$,
\begin{eqnarray}
d_{21}(N)=Na_{11}(I_\mathcal{A})=b_{22}(I_\mathcal{B})N\label{30103}
\end{eqnarray}
for every $N\in \mathcal{N}$.
\

We assume that (1) holds. The proofs for the other cases are analogous.\

For any $A\in \mathcal{A}$ and $M\in \mathcal{M}$,
$a_{11}(I_\mathcal{A})AM=AMb_{22}(I_\mathcal{B})=Aa_{11}(I_\mathcal{A})M$.
Since ${\mathcal M}$ is a faithful left $\mathcal{A}$-module, we have
\begin{eqnarray*}
a_{11}(I_\mathcal{A})A=Aa_{11}(I_\mathcal{A})
\end{eqnarray*}
whence
\begin{eqnarray}
a_{11}(A)=Aa_{11}(I_\mathcal{A})+\lambda\left(\left[
                                     \begin{array}{cc}
                                       A & 0 \\
                                       0 & 0 \\
                                     \end{array}
                                   \right]\right)I_\mathcal{A}=a_{11}(I_\mathcal{A})A+\lambda\left(\left[
                                     \begin{array}{cc}
                                       A & 0 \\
                                       0 & 0 \\
                                     \end{array}
                                   \right]\right)I_\mathcal{A}.\label{30104}
\end{eqnarray}
For any $B\in \mathcal{B}$ and $M\in \mathcal{M}$,
$MBb_{22}(I_\mathcal{B})=a_{11}(I_\mathcal{A})MB=Mb_{22}(I_\mathcal{B})B$.
Since ${\mathcal M}$ is a faithful right $\mathcal{B}$-module, we have
\begin{eqnarray}
b_{22}(B)=b_{22}(I_\mathcal{B})B+\lambda\left(\left[
                                     \begin{array}{cc}
                                       0 & 0 \\
                                       0 & B \\
                                     \end{array}
                                   \right]\right)I_\mathcal{B}=Bb_{22}(I_\mathcal{B})+\lambda\left(\left[
                                     \begin{array}{cc}
                                       0 & 0 \\
                                       0 & B \\
                                     \end{array}
                                   \right]\right)I_\mathcal{B}.\label{30105}
\end{eqnarray}
\

For any $A\in \mathcal{A}$, $M\in \mathcal{M}$, $N\in \mathcal{N}$ and $B\in \mathcal{B}$,\\
$$\delta\left(\left[\begin{array}{cc}      A        &       M          \\       N        &      B      \\         \end{array}\right]\right)=
\left[\begin{array}{cc}  a_{11}(A)+\lambda\left(\left[
                                     \begin{array}{cc}
                                       0 & M \\
                                       N & B \\
                                     \end{array}
                                   \right]\right)I_\mathcal{A}     &     c_{12}(M)    \\   d_{21}(N)    &  b_{22}(B)+\lambda\left(\left[
                                     \begin{array}{cc}
                                       A & M \\
                                       N & 0 \\
                                     \end{array}
                                   \right]\right)I_\mathcal{B}  \\         \end{array}\right],$$
$$\delta(I)\left[\begin{array}{cc}      A        &       M          \\       N        &      B      \\         \end{array}\right]
=\left[\begin{array}{cc}  a_{11}(I_\mathcal{A})A  &  a_{11}(I_\mathcal{A})M \\   b_{22}(I_\mathcal{B})N &  b_{22}(I_\mathcal{B})B \\ \end{array}\right]$$
and
$$\left[\begin{array}{cc}      A        &       M          \\       N        &      B      \\         \end{array}\right]\delta(I)=\left[\begin{array}{cc}   Aa_{11}(I_\mathcal{A})  &  Mb_{22}(I_\mathcal{B}) \\   Na_{11}(I_\mathcal{A}) & Bb_{22}(I_\mathcal{B}) \\ \end{array}\right].$$
So by (\ref{30102})-(\ref{30105}), we have for every $S\in \mathcal{U}$,
$$\delta(S)=\delta(I)S+\lambda(S)=S\delta(I)+\lambda(S).$$
The remaining part goes along the same line as the proof of Corollary \ref{205} and this completes the proof.
\end{proof}

Note that a unital prime ring $\mathcal{A}$ with a non-trivial idempotent $P$ can be written as the matrix form
$\left[
  \begin{array}{cc}
    P\mathcal{A}P & P\mathcal{A}P^\perp \\
    P^\perp\mathcal{A}P & P^\perp\mathcal{A}P^\perp \\
  \end{array}
\right]$.
Moreover, for any $A\in \mathcal{A}$, $PAP\mathcal{A}(I-P)=0$ and $P\mathcal{A}(I-P)A(I-P)$ imply $PAP=0$ and $(I-P)A(I-P)=0$, respectively.

\begin{corollary}\label{303}
Let $\mathcal{A}$ be a unital prime ring with a non-trivial idempotent $P$.
If $\delta$ is a weak $(m,n,l)$-Jordan centralizer from $\mathcal{A}$ into itself, then $\delta$ is a centralizer.
\end{corollary}

As von Neumann algebras have rich idempotent elements and factor von Neumann algebras are prime, the following corollary is obvious.
\begin{corollary}\label{304}
Let $\mathcal{A}$ be a factor von Neumann algebra and $\delta$ be a weak $(m,n,l)$-Jordan centralizer from $\mathcal{A}$ into itself,
then $\delta$ is a centralizer.
\end{corollary}

Obviously, when $\mathcal{N}=0$, $\mathcal{U}$ degenerates to an upper triangular algebra. Thus we have the following corollary.
\begin{corollary}\label{305}
Let $\mathcal{U}=Tri(\mathcal{A},\mathcal{M},\mathcal{B})$ be an upper triangular algebra
such that $\mathcal{M}$ is a faithful ($\mathcal{A}$,~$\mathcal{B}$)-bimodule.
If $\delta$ is a weak $(m,n,l)$-Jordan centralizer from $\mathcal{A}$ into itself, then $\delta$ is a centralizer.
\end{corollary}
Let $\mathcal{N}$ be a nest on a Hilbert space $H$ and alg$\hspace{1pt}$$\mathcal{N}$ be the associated algebra.
If $\mathcal{N}$ is trivial, then alg$\hspace{1pt}$$\mathcal{N}$ is $B(H)$. If $\mathcal{N}$ is nontrivial, take a nontrivial projection $P\in \mathcal{N}$.
Let $\mathcal{A}=P\mathrm{alg}\hspace{1pt}\mathcal{N}P$, $\mathcal{M}=P\mathrm{alg}\hspace{1pt}\mathcal{N}(I-P)$ and $\mathcal{B}=(I-P)\mathrm{alg}\hspace{1pt}\mathcal{N}(I-P)$.
Then $\mathcal{M}$ is a faithful ($\mathcal{A}$, $\mathcal{B}$)-bimodule, and alg$\hspace{1pt}$$\mathcal{N}$=Tri($\mathcal{A}$, $\mathcal{M}$, $\mathcal{B}$)
is an upper triangular algebra.
Thus as an application of Corollaries \ref{304} and \ref{305}, we have the following corollary.
\begin{corollary}
Let $\mathcal{N}$ be a nest on a Hilbert space $H$ and $\mathrm{alg}\hspace{1pt}\mathcal{N}$ be the associated algebra.
If $\delta$ is a weak $(m,n,l)$-Jordan centralizer from $\mathrm{alg}\hspace{1pt}\mathcal{N}$ into itself, then $\delta$ is a centralizer.
\end{corollary}

In the following, we study $(m,n,l)$-Jordan centralizer on AF algebras. A unital $C^*$-algebra $\mathcal{B}$ is called \emph{approximately finite} (AF) if $\mathcal{B}$ contains an increasing chain $\mathcal{B}_n\subseteq \mathcal{B}_{n+1}$ of finite-dimensional $C^*$-subalgebra, all containing the unit $I$ of $\mathcal{B}$, such that $\bigcup_{n=1}^\infty\mathcal{B}_n$ is dense in $\mathcal{B}$. For more details and related terms, we refer the readers to \cite{Derivable}.

\begin{lemma}\label{306}
Let $\mathcal{M}_n(\mathbb{C})$ be the set of all $n \times n$ complex matrices,
$\mathcal{A}$ be a CSL subalgebra of
$\mathcal{M}_{n_1}(\mathbb{C}) \oplus \cdots \oplus \mathcal{M}_{n_k}(\mathbb{C})$, and
$\mathcal{B}$ be an algebra such that $\mathcal{M}_{n_1}(\mathbb{C})
\oplus \cdots \oplus \mathcal{M}_{n_k}(\mathbb{C}) \subseteq \mathcal{B}$ as
an embedding. If $\delta$ is an $(m,n,l)$-Jordan centralizer from $\mathcal{A}$
into $\mathcal{B}$, then $\delta$ is a centralizer.
\end{lemma}

\begin{proof}
Let $\mathcal{A}$ be the linear span of its matrix units $\{E_{ij}\}$, and since $\delta$ is linear, we only need to show that for any $i$, $j$,
\begin{eqnarray}
\delta(E_{ij})=E_{ij}\delta(I)=\delta(I)E_{ij}.\label{300003}
\end{eqnarray}

If $i=j$, by Lemma \ref{202}, (\ref{300003}) is clear.\

Next, we will prove (\ref{300003}) for $i\neq j$. By Lemma \ref{201} and Remark \ref{211}, we have
\begin{eqnarray*}
(m+n+l)\delta(E_{ij})&=&(m+n+l)\delta(E_{ii}E_{ij}+E_{ij}E_{ii})\\
&=&m\delta(E_{ii})E_{ij}+nE_{ii}\delta(I)E_{ij}+lE_{ii}\delta(I)E_{ij}\\
&=&(m+n+l)\delta(E_{ii})E_{ij},
\end{eqnarray*}
Hence $\delta(E_{ij})=\delta(E_{ii})E_{ij}$ for any $i$, $j$.\

Similarly, we have $\delta(E_{ij})=E_{ij}\delta(E_{jj})$ for any $i$, $j$.\

Hence for any $i$, $j$,
\begin{eqnarray*}
E_{ij}\delta(I)=E_{ij}\sum_{k=1}^n \delta(E_{kk})=E_{ij}\sum_{k=1} ^n E_{kk}\delta(E_{kk})=E_{ij}\delta(E_{jj})=\delta(E_{ij}).
\end{eqnarray*}

Similarly, we have for any $i$, $j$, $\delta(I)E_{ij}=\delta(E_{ij})$ and the proof is complete.
\end{proof}

\begin{theorem}
Let $\mathcal{A}$ be a canonical subalgebra of an AF $C^*$-algebra $\mathcal{B}$. If $\delta$ is a bounded $(m,n,l)$-Jordan centralizer from $\mathcal{A}$ into $\mathcal{B}$, then $\delta$ is a centralizer.
\end{theorem}

$Proof$. Suppose $\delta$ is a bounded $(m,n,l)$-Jordan centralizer from $\mathcal{A}$ into $\mathcal{B}$.
Since $\mathcal{A}_n$ is a CSL algebra, $\delta |_{\mathcal{A}_n}$ is a centralizer by Lemma \ref{306}; that is,
for any $S$ in $\mathcal{A}_n$, $$\delta(S)= \delta(I)S=S\delta(I).$$
Since $\delta$ is norm continuous and $\cup _{i=1} ^\infty A_n$ is
dense in $A$, it follows that $\delta$ is a
centralizer.

\section*{Acknowledgement}
This work is supported by NSF of China.


\begin{thebibliography}{99}

\bibitem{CSL} Davidson, K. (1988). Nest Algebras, \textit{Pitman Research Notes in Mathematics Series 191}.

\bibitem{submit} Guo, J., Li, J. On centralizers of some reflexive algebras, submitted.

\bibitem{Hadwin} Hadwin, D., Li, J. (2008). Local derivations and local automorphisms on some algebras, \textit{J. Operator Theroy} 60(1):29-44.

\bibitem{lambrou} Lambrou, M. (1990). On the rank of operators in reflexive algebras, \textit{Linear Algebra Appl.} 142:211-235.

\bibitem{Derivable} Li, J., Pan, Z.(2011).  On derivable mappings, \textit{J. Math. Anal. Appl.} 374(1): 311-322.

\bibitem{ORO} Longstaff, W. (1976). Operators of rank one in reflexive algebras, \textit{Can. J. Math.} 28:19-23.

\bibitem{SRL} Longstaff, W. (1975). Strongly reflexive lattices, \textit{J. London Math. Soc.} 11:491-498.

\bibitem{mbase} Longstaff, W., Panaia, O. (2000). $\mathcal{J}$-subspaces and subspace M-bases, \textit{Stud. Math.} 139:197-212.

\bibitem{LDRA} Lu, F., Liu, B. (2009). Lie derivations of reflexive algebras, \textit{Integr. Equ. Oper. Theory} 64:261-271.

\bibitem{JDRA} Lu, F. (2010). Jordan derivations of reflexive algebras, \textit{Integr. Equ. Oper. Theory} 67:51-56.

\bibitem{JDCSL} Lu, F. (2009). The Jordan structure of CSL algebras, \textit{Stud. Math.} 190:283-299.

\bibitem{Power} Power, S. (1992). Limit algebras: an introduction to subalgebras of $C^*$-algebras, \textit{Pitman Research Notes in Mathematics Series 278}.
\bibitem{Samei} Samei, E. (2005). Approximately local derivations, \textit{J. Londan Math. Soc.} 71:759-778.

\bibitem{sands} Sands, A. (1973). Radicals and Morita contexts, \textit{J. Algebra} 24:335-345.

\bibitem{CSR} Vukman, J. (1999). An identity related to centralizers in semiprime rings, \textit{Comment. Math. Univ. Carolinae} 40:447-456.

\bibitem{CORA} Vukman, J., Kosi-Ulbl, I. (2005). Centralizers on rings and algebras, \textit{Bull. Austral. Math. Soc.} 71:225-234.

\bibitem{OCS} Vukman, J. (2007). On centralizers of semisomple $H^*$ algebras, \textit{Taiwan J. Math.} 4:1063-1074.

\bibitem{COSR} Vukman, J. (2001). Centralizers on semiprime rings, \textit{Comment. Math. Univ. Carolinae} 42:237-245.

\bibitem{mn} Vukman, J. (2010). On (m, n)-Jordan centralizers in rings and algebras, \textit{Glasnik Matematicki} 45:43-53.

\bibitem{Zalar} Zalar, B. (1991).  On centralizers of semiprime rings, \textit{Comment. Math. Univ. Carolinae} 32:609-614.



\end{thebibliography}
\end{document}